\documentclass[a4paper, 11pt]{amsart}
\usepackage[lmargin=1.3in,rmargin=1.3in,tmargin=1.5in,bmargin=1.3in]{geometry}
\usepackage{amsmath, amssymb, amsthm, amscd}
\usepackage{hyperref}

\usepackage{amsthm}
\usepackage{mathtools}
\usepackage{hyperref}

\theoremstyle{plain}
\newtheorem{theorem}{Theorem} [section] 
\newtheorem{lemma}[theorem]{Lemma}
\newtheorem{proposition}[theorem]{Proposition}
\newtheorem{corollary}[theorem]{Corollary}

\theoremstyle{definition}
\newtheorem{definition}[theorem]{Definition}

\theoremstyle{remark}
\newtheorem*{remark}{Remark}
\newtheorem*{question}{Question}
\newtheorem{example}[theorem]{Example}

\newcommand{\inter}{\mathcal{I}}
\newcommand{\lattice}{\mathcal{L}}

\newcommand{\Q}{\mathbb{Q}}
\newcommand{\Z}{\mathbb{Z}}

\newcommand{\rk}{\text{rk}}

\title[On intersection forms of definite 4-manifolds with boundary]{On intersection forms of definite 4-manifolds bounded by a rational homology 3-sphere}

\author{Dong Heon Choe}
\address{Department of Mathematical Science, Seoul National University, Seoul 08826, Republic of Korea}
\email{honey8276@snu.ac.kr}

\author{Kyungbae Park}
\address{School of Mathematics, Korea Institute for Advanced Study, Seoul 02455, Republic of Korea}
\email{kbpark@kias.re.kr}
\urladdr{newton.kias.re.kr/~kbpark}

\keywords{Smooth 4-manifolds; intersection forms; spherical 3-manifolds; integral lattices}
\subjclass[2010]{57M27, 57N13, 57R58}

\begin{document}

\begin{abstract}
We show that, if a rational homology 3-sphere $Y$ bounds a positive definite smooth 4-manifold, then there are finitely many negative definite lattices, up to the stable-equivalence, which can be realized as the intersection form of a smooth 4-manifold bounded by $Y$. To this end, we make use of constraints on definite forms  bounded by $Y$ induced from Donaldson's diagonalization theorem, and correction term invariants due to Fr\o yshov, and Ozsv\'ath and Szab\'o. In particular, we prove that all spherical 3-manifolds satisfy such finiteness property. 
\end{abstract}
\maketitle

\section{Introduction}
Throughout this paper we assume that all manifolds are compact and oriented. We say a $4$-manifold $X$ \emph{is bounded by} a $3$-manifold  $Y$ if $Y$ is homeomorphic to the boundary of $X$ and the orientation of $Y$ inherits the orientation of $X$ in the standard way. 

The intersection pairing $Q_X$ on $H_2(X;\Z)/Tors$ of a 4-manifold $X$ is an integer-valued, symmetric, bilinear form over $\Z^{b_2(X)}$. If the boundary of $X$ is a rational homology 3-sphere or empty, $Q_X$ is nondegenerate. In this case, the intersection form $Q_X$ algebraically forms a lattice. In particular, if $X$ is closed, then $Q_X$ is unimodular, i.e.\ $\det(Q_X)=\pm1$.  

The remarkable works of Donaldson and Freedman in the early 1980's portray a big difference between topological and smooth categories in dimension 4 for the answer to the following question:
\begin{question}
	Which negative definite, unimodular lattices are realized as the intersection form of a closed 4-manifold?
\end{question} 
Freedman showed that any unimodular definite lattice can be realized as the intersection form of a closed, \emph{topological} 4-manifold \cite{Freedman:1982}. On the other hand, Donaldson's diagonalization theorem asserts that the only standard diagonal one can be realized as the intersection form of a closed, \emph{smooth} 4-manifold  \cite{Donaldson:1983, Donaldson:1987}.

In this paper, we would like to study this phenomenon for 4-manifolds with a fixed boundary. We say a lattice $\Lambda$ \emph{is smoothly (resp. topologically) bounded by} a 3-manifold $Y$ if $\Lambda$ can be realized as the intersection form of a smooth (resp. topological) 4-manifold with the boundary $Y$. It follows easily by connected summing $\overline{\mathbb{CP}^2}$ to a 4-manifold realizing $\Lambda$ that if $Y$ bounds a lattice $\Lambda$, then it also does $\Lambda\oplus\langle-1\rangle$. We call this procedure the \emph{stabilization} of $\Lambda$.
\begin{definition}
	Two negative definite lattices $\Lambda_1$ and $\Lambda_2$ are \emph{stable-equivalent} if $\Lambda_1\oplus\langle-1\rangle^m\cong\Lambda_2\oplus\langle-1\rangle^n$ for some non-negative integers $m$ and $n$.
\end{definition}
Let $\inter(Y)$ (resp. $\inter^{TOP}(Y)$) denote the set of all negative definite lattices that can be smoothly (resp. topologically) bounded by $Y$, up to the stable-equivalence. In terms of these notations, aforementioned Freedman and Donaldson's results can be interpreted as 
\[
	\inter^{TOP}(S^3)=\{[\Lambda]\mid\Lambda\colon \text{any unimodular negative definite lattice}\}\\
\]
and
\[
	\inter(S^3)=\{[\langle-1\rangle]\}.
\]

Following the Freedman's results, Boyer studied a realization problem for topological 4-manifolds with a fixed boundary $Y$ \cite{Boyer:1986-1}. Roughly speaking, any forms presenting the linking pairing of $Y$ can be realized. Hence one can easily observe the following:
\begin{theorem}
	Let $Y$ be a rational homology 3-sphere. Then 
	\[|\inter^{TOP}(Y)|=\infty.\] 
\end{theorem}
\begin{proof}
	For the linking pairing of a 3-manifold $Y$, it is known by Edmonds \cite[Section 6]{Edmonds} that there is a definite form presenting it. This form is realized by a topological 4-manifold $W$ by Boyer's result. Now, we have infinitely many  definite lattices bounded by $Y$, up to the stabilization, by connected summing $W$ and closed topological 4-manifolds with non-standard definite intersection forms.
\end{proof}

The main interest of this article is whether the finiteness of $\inter(Y)$ holds for any rational homology 3-sphere $Y$, similarly to the case of $S^3$. Although we conjecture that $\inter(Y)$ is finite for any rational homology 3-sphere $Y$, it has been known only for manifolds with small \emph{correction term} invariant. Recall that the correction term is a rational-valued invariant for rational homology 3-spheres, due to Fr\o yshov in Seiberg-Witten theory \cite{Froyshov:1996} and Ozsv\'ath and Szab\'o in Heegaard Floer theory \cite{Ozsvath-Szabo:2003}. The correction term is known to give constraints on definite lattices smoothly bounded by $Y$. Let $\Lambda$ be a negative definite lattice bounded by $Y$. Then the following inequality is satisfied
\[
	\delta(\Lambda)\leq d(Y),
\]
where $\delta$ is a rational-valued invariant for definite lattices defined in \cite{Elkies}: see also Section \ref{sec:preliminary}. On the other hand, the finiteness of the number of stable classes of definite unimodular lattices are known only for small $\delta\leq6$, purely algebraically \cite{Elkies, Elkies2, Gaulter, Nebe_Venkov}. Therefore, we can conclude that $\inter(Y)$ is finite for any integral homology 3-sphere $Y$ with $d(Y)\leq6$. For rational homology spheres with sufficiently small $\delta$, we deduce a similar finiteness of $\inter(Y)$ from the result of Owens and Strle in \cite{Owens_Strle} about non-unimodular lattices. Our main results are that $\inter(Y)$ is finite under some additional conditions.
\begin{theorem}\label{thm:finite}
	Let $Y_1$ and $Y_2$ be rational homology 3-spheres. If there is a negative definite smooth cobordism from $Y_1$ to $Y_2$ and $|\inter(Y_2)|<\infty$, then  $|\inter(Y_1)|<\infty$.
\end{theorem}

Since we know $|\inter(S^3)|<\infty$ by Donaldson's diagonalization theorem, we have the following corollary.

\begin{corollary}\label{cor:finite}
	Let $Y$ be a rational homology 3-sphere. If $Y$ bounds a positive (resp. negative) definite smooth 4-manifold, then there are only finitely many negative (resp. positive) definite lattices, up to the stable-equivalence, which can be realized as the intersection form of a smooth 4-manifold bounded by $Y$. 
	
	In other words, if $\inter(-Y)\neq\emptyset$, then $|\inter(Y)|<\infty$.
\end{corollary}
\begin{proof}
	Let $W$ be a positive definite smooth 4-manifold with the boundary $Y$. Then we construct a negative definite cobordism from $Y$ to $S^3$ by reversing the orientation of the punctured $W$.
\end{proof} 

\begin{remark}
	We say a 4-manifold $X$ is negative (resp. positive) definite if $b_2(X)=b_2^-(X)$ (resp. $b_2(X)=b_2^+(X)$). In particular, a 4-manifold with $b_2=0$ is considered to be both positive and negative definite.
\end{remark}

In \cite[Corollary 0.4]{Boyer:1986-1}, Boyer showed that there are only finitely many homeomorphism types of simply-connected 4-manifolds which have given intersection form and boundary. Our results give a bit of answer to the geography problem of simply-connected, smooth 4-manifolds with a fixed boundary.

\begin{corollary}
	If $Y$ bounds a positive (resp. negative) definite smooth 4-manifold, then there are finitely many homeomorphism types of simply-connected, negative (resp. positive) definite smooth 4-manifolds bounded by $Y$, up to the stabilization.
\end{corollary}

To prove our main theorem, we consider a set of lattices $\lattice$ defined purely algebraically in terms of invariants of a given 3-manifold $Y$. This set $\lattice$ contains negative definite lattices, up to the stable-equivalence, which satisfy the conditions for a definite lattice to be smoothly bounded by $Y$, induced from the correction term invariants and fundamental obstructions from the algebraic topology. See Section \ref{sec:restrictions} for details of these conditions. Then Theorem \ref{thm:finite} readily follows after we show the finiteness of $\lattice$.

\begin{theorem}\label{thm:alg}
	Let $\Gamma_1$ and $\Gamma_2$ be fixed negative definite lattices, and $C>0$ and $D\in\Z$ be constants. Define  $\lattice(\Gamma_1,\Gamma_2; C, D)$ to be the set of negative definite lattices $\Lambda$, up to the stable-equivalence, satisfying the following conditions:
	\begin{itemize}
		\item $det(\Lambda)=D,$
		\item $\delta(\Lambda)\leq C,$ and
		\item $\Gamma_1\oplus\Lambda$ embeds into $\Gamma_2\oplus\langle-1\rangle^N$, $N=\rk (\Gamma_1)+\rk(\Lambda)-\rk(\Gamma_2).$
	\end{itemize}
	Then $\lattice(\Gamma_1,\Gamma_2; C, D)$ is finite.
\end{theorem}

Our proof of Theorem \ref{thm:alg} is highly inspired by the work of Owens and Strle in \cite{Owens_Strle}, where they studied non-unimodular lattices in terms of the  lengths of characteristic covectors. The key idea of our proof is to improve one of their inequalities on the length of a characteristic covector, enough to give an upper bound on the rank of the lattices in $\lattice(\Gamma_1, \Gamma_2; C, D)$.

\begin{remark}
	We remark that not all lattices which represent a stable class in $\inter(Y)$ can be realized as the intersection form of a smooth $4$-manifold bound by $Y$. It is well known that the Poincar\'e homology 3-sphere $\Sigma$, oriented as the boundary of the $-E_8$-plumbed $4$-manifold, can be obtained by $(-1)$-framed surgery along the left-handed trefoil knot. Hence $[\langle-1\rangle] = [\emptyset] \in \inter(\Sigma)$. Whereas one can prove, using a constraint from the Donaldson's diagonalization theorem, $\Sigma$ cannot bound any $4$-manifold with $b_2=0$, i.e.\ the empty lattice cannot be realized. See Example \ref{ex:poincare}.
\end{remark}

\subsection*{Seifert fibered spaces such that $\inter(Y)<\infty$}
It is interesting to ask which 3-manifolds satisfy the condition in Corollary \ref{cor:finite}, i.e.\ which 3-manifolds bound a positive definite smooth 4-manifold. Since many of 3-manifolds, including all Seifert fibered rational homology 3-spheres, bound a definite smooth 4-manifold up to the sign (see Proposition \ref{prop:both}), it is more reasonable to find families of 3-manifolds bounding definite smooth 4-manifolds of both signs.

It is well known that any lens space satisfy such property. Note that lens spaces can be obatained by double covering of $S^3$ branched along 2-bridge knots. As generalizing lens spaces to this direction, the double branched covers of $S^3$ along quasi-alternating links are known to bound definite smooth 4-manifolds of both signs (with trivial first homology) \cite[Proof of Lemma 3.6]{Ozsvath-Szabo:2005-1}. Also notice that lens spaces can be obtained by Dehn-surgery along the unknot. In \cite{Owens-Strle:2012-1}, Owens and Strle classified 3-manifolds obtained by Dehn-surgery on torus knots which can bound definite smooth 4-manifolds of both signs.

In Section \ref{sec:Seifert}, we consider this question for another class of 3-manifolds, Seifert-fibered rational homology 3-spheres, which also contains lens spaces. In particular, we completely determine which spherical 3-manifolds bound definite smooth 4-manifolds of both signs, and finally show that any spherical 3-manifold $Y$ has the property that $|\inter(Y)|<\infty$.

\begin{theorem}
	\label{thm:spherical}
	Let $Y$ be a spherical 3-manifold. Then, there are finitely many stable classes of negative definite lattices which can be realized as the intersection form of a smooth 4-manifold bounded by $Y$, i.e.\ $|\inter(Y)|<\infty.$
\end{theorem}

\subsection*{Further questions}

\subsubsection*{Generalizing our results to 3-manifolds with $b_1>0$} Note that a 3-manifold with $b_1>0$ might bound a 4-manifold with degenerate intersection form. Hence, we generalize $\inter(Y)$ to be the set of stable classes of lattices which can be realized as the maximal nondegenerate subspace of the intersection form of a negative semi-definite 4-manifolds bounded by a 3-manifold $Y$. 

On the other hand, the correction term invariants for rational homology 3-spheres has been generalized for some 3-manifolds with $b_1>0$ by Levine and Ruberman in \cite{Levine-Ruberman:2014}, and by Behrens and Golla in \cite{Behrens-Golla:2018-1}. These generalized correction term invariants also give restrictions on semi-definite intersection forms of $4$-manifolds bounded by a $3$-manifold. Hence we expect the analogous finiteness results of $\inter(Y)$ for any closed oriented 3-manifold $Y$ under suitable conditions, but we leave this for the future study. 

\subsubsection*{Determining the order of $\inter$} 
Although we know $|\inter(S^3)|=1$ from Donaldson's diagonalization theorem, determining the exact order of $\inter(-)$ for other than $S^3$ seems a difficult problem. For instance, the correction term of the Poincar\'e homology sphere $\Sigma$ is known as $2$. On the other hand, in \cite{Elkies2} Elkies showed that there are only $15$ stable classes of negative definite, unimodular lattices with $\delta\leq2$. Since the trivial lattice and $-E_8$ lattice are in fact realized to be bounded by $\Sigma$, we have $2\leq|\inter(\Sigma)|\leq 15$. Then what is the exact value of $|\inter(\Sigma)|$?  

\section{Preliminary}\label{sec:preliminary}
In this section we collect some background materials that will be used to prove our main theorems. 

\subsection{Lattices}
A \emph{lattice} of rank $n$ is a free abelian group $\mathbb{Z}^n$ equipped with an integer-valued, nondegenerate, symmetric, bilinear form, \[Q\colon\mathbb{Z}^n\times\mathbb{Z}^n\rightarrow\mathbb{Z}.\] Let $\Lambda=(\mathbb{Z}^n, Q)$ be a lattice. By tensoring $\Lambda$ with $\mathbb{R}$, one can extend $Q$ to a symmetric bilinear form over the vector space $\mathbb{R}^n$. We define the \emph{signature} of $\Lambda$ to be the signature of $Q$. We say $\Lambda$ is \emph{positive} (resp. \emph{negative} ) \emph{definite} if the signature of $\Lambda$ equals to the (resp. negative) rank of $\Lambda$. 

By fixing a basis $\{v_1,\dots, v_n\}$ for $\Lambda$, we can represent $\Lambda$ by an $n$ by $n$ matrix, $[Q(v_i, v_j)]$. The \emph{determinant} of a lattice $\Lambda$, $\det(\Lambda)$, is the determinant of a matrix representation of $\Lambda$. In particular, if the determinant of a lattice is $\pm1$, or equivalently a corresponding matrix is invertible over $\mathbb{Z}$, we say the lattice is \emph{unimodular}.

The dual lattice $\Lambda^*\coloneqq Hom(\Lambda, \mathbb{Z})$ can be identified with the set of elements in $\xi\in\Lambda\otimes\mathbb{R}$ such that $\xi\cdot w\in\mathbb{Z}$ for any $w\in\Lambda$. We call $\xi$ in $\Lambda^*$ a \emph{characteristic covector} if $\xi\cdot w\equiv w\cdot w$ modulo 2 for any $w\in\Lambda$. We say a lattice $\Lambda_1$ \emph{embeds} into $\Lambda_2$ if there is a monomorphism from $\Lambda_1$ to $\Lambda_2$ preserving bilinear forms. 

The lattice $\langle p\rangle$ denotes the lattice of rank 1 represented by the matrix $[p]$. The \emph{standard negative definite lattice} of rank $n$ is the lattice $\langle-1\rangle^n$, the direct sum of $n$ copies of $\langle-1\rangle$. For a negative definite lattice $\Lambda$ with rank $n$, we define $\delta$-invariant of $\Lambda$ as
\begin{equation*}
	\delta(\Lambda)\coloneqq\frac{n-\min_{\xi\in Char(\Lambda)}\left|\xi\cdot\xi\right|}{4},
\end{equation*}
where $Char(\Lambda)$ is the set of characteristic covectors of $\Lambda$. Note that a negative definite lattice $\Lambda$ can be uniquely decomposed as $\Lambda'\oplus\langle-1\rangle^m$ so that $\Lambda'$ dose not contain any vector with square $-1$, and $\delta(\Lambda)=\delta(\Lambda')$.
\subsection{Restrictions on lattices bounded by a rational homology 3-sphere}
\label{sec:restrictions}

We recall some constraints on lattices bounded by a given rational homology 3-sphere. We also refer the readers to Owens and Strle's survey paper \cite{Owens_Strle_survey} for more detail.

\subsubsection*{A topological obstruction}
Suppose $Y$ bounds a lattice $\Lambda=(\Z^n,Q)$ and $X$ is a 4-manifold realizing $\Lambda$. From the homology long exact sequence of the pair $(X,Y)$, we have
\begin{equation*}
	|H^2(Y;\mathbb{Z})|=|\det(\Lambda)|t^2
\end{equation*} 
for some integer $t$. See \cite[Lemma 2.1]{Owens_Strle3} for example. In particular, $\det(\Lambda)$ divides $|H^2(Y;\mathbb{Z})|$. We remark that there are more detailed conditions regarding the linking form of $\Lambda$ and the linking pairing of $Y$, but we only recall the above simple property for our purpose.

\subsubsection*{Donaldson's Theorem} The celebrated Donaldson's diagonalization theorem can be used to give a constraint on definite lattices smoothly bounded by a 3-manifold. Recall the Donaldson's theorem. 
\begin{theorem}[{\cite[Theorem 1]{Donaldson:1987}}]
	Suppose $X$ is a closed, smooth 4-manifold. If the intersection form of $X$ is negative definite, then it is isometric to the standard definite lattice $(\Z^n, \langle-1\rangle^n)$.
\end{theorem}
Suppose a rational homology 3-sphere $Y$ bounds a positive definite, smooth 4-manifold $W$. If $Y$ bounds a negative definite lattice $\Lambda$, then we can construct a negative definite, closed, smooth 4-manifold by gluing a 4-manifold $X$ realizing $\Lambda$ with $W$ along $Y$. By Donaldson theorem, the intersection form of the closed smooth 4-manifold $X\cup_Y -W$ is the standard negative definite of the rank,  $\rk(\Lambda)+\rk(Q_W)$. Then, from the Mayer-Vietoris sequence for the pair $(X\cup_Y -W, Y)$, 
\begin{equation*}
	\dots\rightarrow H_2(Y)\rightarrow H_2(X)\oplus H_2(-W)\rightarrow H_2(X\cup_Y -W)\rightarrow\dots,
\end{equation*}
we have an embedding of $\Lambda\oplus-Q_W$ into $\langle-1\rangle^{\rk(\Lambda)+\rk(Q_W)}$. 
\begin{example}\label{ex:poincare}
	The Poincar\'e homology sphere $-\Sigma$, oriented as the boundary of $E_8$-plumbed 4-manifold, naturally bounds the $E_8$ lattice which is positive definite. It is well known that $-E_8$ lattice cannot be embedded into the standard negative definite lattice: see \cite[Lemma 3.3]{Lecuona_Lisca_Stein} for example. Therefore, $-\Sigma$ cannot bound any negative definite smooth 4-manifold (including a 4-manifold with trivial intersection form), i.e. $\inter(-\Sigma)=\emptyset$. 
\end{example}

\subsubsection*{Correction terms}\label{sec:OS}
Let $Y$ be a rational homology 3-sphere and $\mathfrak{t}$ be a spin$^c$ structure over $Y$. In \cite{Ozsvath-Szabo:2003}, Ozsv\'ath and Szab\'o defined a rational valued invariant  for $(Y,\,\mathfrak{t})$ called \emph{the correction term} or \emph{d-invariant}, denoted by $d(Y,\,\mathfrak{t})$. It is an analogous invariant to Fr\o yshov's in Seiberg-Witten theory \cite{Froyshov:1996}. Among many important properties of the correction term, it gives a constraint on a definite 4-manifold bounded by $Y$. 

\begin{theorem}[{\cite[Theorem 9.6]{Ozsvath-Szabo:2003}}]	 \label{thm:OS}
	If $X$ is a negative definite, smooth $4$-manifold bounded by $Y$, then for each spin$^c$ structure $\mathfrak{s}$ over $X$
	\begin{equation*}        
		c_1(\mathfrak{s})^2+n\leq 4d(Y,\,\mathfrak{s}|_Y),
	\end{equation*}
	where $c_1(-)$ denotes the first Chern class,  $n$ is the rank of $H_2(X;\,\mathbb{Z})$, and $\mathfrak{s}|_Y$ is the restriction of $\mathfrak{s}$ over $Y$.
\end{theorem}

Since any characteristic covector in $H^2(X;\Z)/Tors$ is identified with $c_1(\mathfrak{s})$ for a spin$^c$-structure $\mathfrak{s}$ on $X$, we have the following.
\begin{proposition}\label{prop:d_inequality}
	If a negative definite lattice $\Lambda$ is smoothly bounded by a rational homology 3-sphere $Y$, then 
	\begin{equation*}\label{eq:os}
		\delta(\Lambda)\leq d(Y),
	\end{equation*}
	where $$d(Y)\coloneqq\max_{\mathfrak{t}\in \text{Spin}^c(Y)}d(Y,\mathfrak{t}).$$
\end{proposition}

\section{Finiteness of the number of definite lattices bounded by a rational homology 3-sphere}\label{sec:finite}
The purpose of this section is to prove Theorem \ref{thm:alg} and consequently Theorem \ref{thm:finite}. First, recall the following well-known fact: see \cite[p. 18]{Milnor-Husemoller:1973} for example.
\begin{lemma}\label{lem:finite}
	There are finitely many isomorphism classes of definite lattices which have a given rank and determinant.
\end{lemma}
We first show a special case of Theorem \ref{thm:alg} in which $\Gamma_1$ and $\Gamma_2$ are the trivial empty lattices.
\begin{proposition}\label{prop:finite}
	Let $C>0$ and $D\in \Z$ be constants. There are finitely many negative definite lattices $\Lambda$, up to the stable-equivalence, which satisfy the following conditions:
	\begin{itemize}
		\item $\det\Lambda=D$,
		\item $\delta(\Lambda)\leq C$, and
		\item $\Lambda$ embeds into $\langle-1\rangle^{\rk(\Lambda)}$ with prime index.
	\end{itemize}
\end{proposition}

\begin{proof}
	Let $\Lambda$ be a negative definite lattice of rank $n$ that satisfies the conditions above. Without loss of generality, we may assume that there is no vector of square $-1$ in $\Lambda$. By Lemma \ref{lem:finite} the theorem follows if we find an upper bound for the rank of $\Lambda$, only depending on $D$ and $C$. Let $\{e,e_1,\dots,e_{n-1}\}$ be the standard basis of the standard negative definite lattice $\langle-1\rangle^n$, i.e. $e^2=-1$, $e\cdot e_i=0$, $e_i\cdot e_j=-\delta_{ij}$ for $i, j=1,\dots,n-1$. 
	
	Let $p$ be the index of the embedding $\iota$ of $\Lambda$ into  $\langle-1\rangle^n$. If $p=1$, then the embedding is an isomorphism and so $\Lambda$ should be the empty lattice, up to the stable-equivalence. Now suppose $p$ is an odd prime. Then the cokernel of $\iota$ is a cyclic group of order $p$, and it is generated by $e+\Lambda$ since $e\notin \Lambda$. Observe that $-e_i\in s_ie+\Lambda$ for some odd integer $s_i\in[-p+1,p-1]$.
	Consider a set of elements of $\Lambda$,
	\[\mathcal{B}\coloneqq\{pe,e_1+s_1e,\dots,e_{n-1}+s_{n-1}e\}.\]
	Since the determinant of the coordinates matrix corresponding to the above set equals to $p$, it is in fact a basis for $\Lambda$. The matrix representation of $\Lambda$ with respect to the basis is given as
	\begin{equation*}
		Q=-\left(\begin{array}{ccccc}
			p^2 & ps_1    & ps_2 & \dots &ps_{n-1}\\
			ps_1    & 1+s_1^2 & s_1s_2 & \dots &s_1s_{n-1}\\
			ps_2    & s_1s_2    & 1+s_2^2 & \ddots & \vdots\\
			\vdots       & \vdots     & \ddots     & \ddots&  s_{n-2}s_{n-1}\\
			ps_{n-1}&s_1s_{n-1}&\dots&s_{n-2}s_{n-1}&1+s_{n-1}^2
		\end{array} \right).
	\end{equation*}
	We also compute the inverse of $Q$ as follows
	\begin{equation*}
		Q^{-1}
		=\left(\begin{array}{c|cccccc}
			-\frac{1+\sum_{i=1}^{n-1}s_i^2}{p^2}&\frac{s_1}{p}&\frac{s_2}{p}& \dots&\frac{s_{n-1}}{p}\\
			\hline
			\frac{s_1}{p}&  &  &  &\\
			\frac{s_2}{p}&  &&    &\\
			\vdots       &  &  & -I_{(n-1)\times(n-1)} & \\
			\frac{s_{n-1}}{p}& & & &
		\end{array} \right).
	\end{equation*}
	Note that $Q^{-1}$ represents the dual lattice of $\Lambda$ with respect to the dual basis of $\mathcal{B}$. Hence a characteristic covector $\xi$ of $\Lambda$ can be written as a vector 
	\begin{equation*}\xi=(k,k_1,\dots,k_{n-1}),\end{equation*} 
	where $k$ is an odd integer and $k_i$'s are even integers, in terms of the dual basis of $Q$ since $p^2$ is odd and $1+s_i^2$ is even for each $i$. From the matrix $Q^{-1}$ we compute \[|\xi\cdot\xi|=\frac{1}{p^2}(k^2+\sum_{i=1}^{n-1}(ks_i-pk_i)^2).\]
	Applying Lemma $\ref{lemma:app}$ below,
	\[\min\{|\xi\cdot\xi|:\xi\text{ characteristic covector of } \Lambda\} \leq \frac{n+2}{3}.\]
	Therefore, by Proposition \ref{prop:d_inequality}, \[\delta(\Lambda)=\frac{n-\min_{\xi\in Char(\Lambda)}\left|\xi\cdot\xi\right|}{4}\leq C,\] 
	and we conclude that \[n\leq 6C+1.\]
	
	In the case of $p=2$, the lattice $\Lambda$ admits a basis 		\[\{2e,e_1+e,\dots,e_{n-1}+e\},\] and hence $0$ vector is characteristic. Therefore, \[\delta(\Lambda)=\frac{n}{4}\leq C.\]
	
	By Lemma \ref{lem:finite}, there are only finitely many negative definite lattices satisfying the given conditions.
\end{proof}

Now, we prove the following algebraic lemma used in the proof above.
\begin{lemma}
	\label{lemma:app}
	For an odd prime $p$ and odd integers $s_1, s_2, \dots, s_{n-1}$ in $[-p+1, p-1]$, there exist an odd integer $k$ and even integers $k_1, k_2, \dots, k_{n-1}$ such that 
	\begin{equation*}\label{eq:alg}
		k^2+\sum_{i=1}^{n-1}(ks_i-pk_i)^2 < \frac{n+2}{3}p^2.
	\end{equation*}
\end{lemma}
\begin{proof}
	For an odd prime $p$, consider the set $K\coloneqq\{-p+2,-p+4,\dots,p-2\}$ of odd integers in the interval $[-p+1, p-1]$. Note, for each $k$ and $s_i$ in $K$, there is a unique even integer $k_i$ such that $ks_i-pk_i\in K$. Denote this $k_i$ by $k_i(k,s_i)$. Since $ks_i\equiv k's_i$ (mod $2p$) implies $k\equiv k'$ (mod $2p$) for $k$ and $k'$ in $K$, we obtain $\{ks_i-p\cdot k_i(k,s_i)|k\in K\}=K$ for each $s_i\in K$. Therefore, 
	\begin{equation*}
		\begin{split}\displaystyle\sum_{k\in K}\displaystyle\sum_{i=1}^{n-1}(ks_i-p\cdot k_i(k,s_i))^2&=(n-1)\cdot2(1^2+3^2+\dots+(p-2)^2)\\
			&=\frac{n-1}{3}p(p-1)(p-2).\end{split}\end{equation*}
	Since $|K|=p-1$, there exists $k\in K$ such that 
	\[ \sum_{i=1}^{n-1}(ks_i-p\cdot k_i(k,s_i))^2\leq \frac{n-1}{3}p(p-2).\]
	Since $|k|<p$, we obtain the desired inequality.
\end{proof}

Now, Theorem \ref{thm:alg} is proved by applying a similar argument of Proposition \ref{prop:finite}. First, observe the following.
\begin{lemma}\label{lem:embedding}
	Let $\Gamma_1$ and $\Gamma_2$ be negative definite lattices. Then the set of stable classes of lattices,  
	\[
		\mathcal{C}(\Gamma_1,\Gamma_2)\coloneqq\{(\text{Im}\iota)^\perp\mid\iota\colon\Gamma_1\hookrightarrow\Gamma_2\oplus\langle-1\rangle^N\text{, an embedding for some } N\in\mathbb{N}\}/{\sim},
	\] 
	is finite.
\end{lemma}
\begin{proof}
	First, we claim that the set \[\{(\text{Im}\iota)^\perp\mid\iota\colon\Gamma_1\hookrightarrow\Gamma_2\oplus\langle-1\rangle^N\text{, an embedding}\}/{\sim}\] is stabilized for some large enough $N$. Let $\{v_1,\dots,v_{\rk(\Gamma_1)}\}$ be a basis for $\Gamma_1$. By considering the representations of $v_i$ in terms of a basis for $\Gamma_2\oplus\langle-1\rangle^N$, if \[N > \sum_{i=1}^{\rk(\Gamma)}|v_i\cdot v_i|,\] then there is a vector $e\in \langle-1\rangle^{N}$ such that $e\cdot e=-1$ and $e\cdot v_i=0$ for all $i=1,\dots,\text{rk}(\Gamma_1)$. Hence in order to prove the finiteness of $\mathcal{C}(\Gamma_1,\Gamma_2)$, it is enough to consider some fixed large $N$. 
	
	Let $\{w_j\}_{j=1}^m$ be a basis for $\Gamma_2\oplus\langle-1\rangle^N$. Note that an embedding of $\Gamma_1$ into $\Gamma_2\oplus\langle-1\rangle^N$ can be presented by the system of equations
	\[v_i=\sum_{j=1}^m a_{i,j}w_j,\]
	where $a_{i,j}$ are integers. Since $\Gamma_1$ and $\Gamma_2\oplus\langle-1\rangle^N$ are definite, the possible choices of $a_{i,j}$ are finite for each $i,j$, and hence the number of possible embedding maps is also finite.
\end{proof}

\begin{proof}[Proof of Theorem \ref{thm:alg}]
	Fix negative definite lattices $\Gamma_1$ and $\Gamma_2$, and constants $C>0$ and $D\in\Z$. Let $\Lambda$ be a negative definite lattice which satisfies the conditions in the theorem. Without loss of generality, we may assume that there is no square $-1$ vector in $\Lambda$. By Lemma \ref{lem:finite}, the theorem follows if we show that the rank of $\Lambda$ is bounded by some constant only depending on $\Gamma_1$, $\Gamma_2$, $C$ and $D$. 
	
	From the third condition of $\Lambda$, there is an embedding 
	\[\iota|_{\Gamma_1}\colon\Gamma_1\hookrightarrow\Gamma_2\oplus\langle-1\rangle^{N},\]
	where $N=\rk(\Lambda)+\rk(\Gamma_1)-\rk(\Gamma_2)$. Let $(\text{Im}(\iota|_{\Gamma_1}))^\perp\cong \langle-1\rangle^n\oplus E$, where $E$ is a lattice without square $-1$ vectors and $n=\rk(\Lambda)-\rk(E)$. Note that $[E]$ is one of the elements in  the finite set $\mathcal{C}(\Gamma_1,\Gamma_2)$ in Lemma \ref{lem:embedding}.
	
	Now we need to find a bound of the rank of $\Lambda$ embedded in $\langle-1\rangle^n\oplus E$. This will be obtained by the similar argument in the proof of Theorem \ref{prop:finite}. The main difference is that we have an extra summand $E$. 
	
	Let $\iota'$ be an embedding $\Lambda$ into $\langle-1\rangle^n\oplus E$, and $p$ be the index of $\iota'$. If $n=0$, i.e.\ the rank of $\Lambda$ is same as the rank of $E$, then we have a bound of the rank of $\Lambda$ by Lemma \ref{lem:embedding}. Similarly, if $p=1$, then $\Lambda \cong \langle-1\rangle^n\oplus E$ and we have the same rank bound of $\Lambda$. 
	
	Now suppose $n\neq 0$ and that $p$ is an odd prime. Let
	\[\{e,e_1,\dots,e_{n-1},f_1,\dots,f_r\}\]
	be a basis for $\langle-1\rangle^n\oplus E$ so that $e^2=-1$, $e\cdot e_i=0$, $e_i\cdot e_j=-\delta_{ij}$ and $e\cdot f_j=e_i\cdot f_j=0$ for any $i,j$ and $E$ is generated by $\{f_1,\dots,f_r\}$. By the same argument in Proposition \ref{prop:finite}, we can choose a basis for $\Lambda$ ,  
	\begin{equation}\label{eq:basis}
	\{pe,e_1+s_1e,\dots,e_{n-1}+s_{n-1}e,f_1+t_1e,\dots,f_r+t_re\}
	\end{equation}
	where $s_i$'s are odd integers in $[-p+1,p-1]$ and $t_j$'s are integers in $[-p+1,p-1]$. Now with respect to the dual coordinates for this basis, write a characteristic covector $\xi$ as 
	\[\xi=(k,k_1,\dots,k_{n-1},l_1,\dots,l_r)\] where $k$ is odd, $k_i$'s are even and $l_j\equiv(f_j+t_je)\cdot(f_j+t_je)$ (mod 2) for each $i,j$. 
	To find the matrices of $\Lambda$ and $\Lambda^{-1}$,
	introduce an $(n+r)\times(n+r)$ matrix $M$ and a $r\times r$ matrix $A$ as 
	\begin{equation*}
		M_{ij}\coloneqq
		\begin{cases}
			p &\text{if } \ i=1,j=1\\
			1 & \text{if } \ i=j, 2\leq j \leq n+r\\
			s_{j-1} & \text{if } \ i=1, 2\leq j\leq n\\
			t_{j-n} & \text{if } \ i=1, n+1\leq j\leq n+r \\
			0 & \text{otherwise},
		\end{cases}
	\end{equation*}
	and
	\begin{equation*}
		A_{ij}\coloneqq f_i\cdot f_j.
	\end{equation*}
	Note that $M$ represents the embedding of $\Lambda$ into $\langle-1\rangle^n\oplus E$ and $A$ represents $E$. By the basis in (\ref{eq:basis}), $\Lambda$ and the dual of $\Lambda$ are represented as follows:
	\begin{equation*}
		Q_{\Lambda}=M^{t}\left(\begin{array}{cc}
			-I_{n\times n} & 0 \\
			0 & A 
		\end{array} \right)M
	\end{equation*}
	and
	\begin{equation}
	\begin{split}
	Q_{\Lambda}^{-1}&=M^{-1}\left(\begin{array}{cc}-I_{n\times n} & 0 \\0 & A^{-1} \end{array} \right)(M^{t})^{-1}\\
	&=-M^{-1}(M^{t})^{-1}+M^{-1}\left(\begin{array}{cc}
	0 & 0 \\
	0 & A^{-1}+I_{r\times r} 
	\end{array} \right)(M^{t})^{-1}.
	\end{split}	
	\label{eq:inverse}
	\end{equation}
	We obtain that \[|\xi\cdot\xi|\leq \frac{1}{p^2}(k^2+\sum_{i=1}^{n-1}(ks_i-pk_i)^2)+|S(k,l_1,\dots ,l_r,t_1,\dots,t_r)|,\]
	for some function $S$. We emphasize that the function $S$ is independent to $n$, since it is obtained from the last term of Equation (\ref{eq:inverse}). Then by Lemma \ref{lemma:app},  
	\[\min\{|\xi\cdot\xi|:\xi\text{ characteristic covector of } L\} \leq \frac{n}{3}+|S|\]
	Therefore, we obtain \[n \leq \frac{3}{2}(4C+|S|)\] from $\delta(\Lambda)\leq C$. The case that $p=2$ is easier to find a similar bound for $n$ by applying the same argument.
	
	For an arbitrary index $p$, we use an idea in \cite{Owens_Strle} to have a sequence of embeddings 
	\[\Lambda=E_0 \hookrightarrow E_1 \hookrightarrow E_2\hookrightarrow \dots \hookrightarrow E_s= \langle-1\rangle^n\oplus E\] 
	such that each embedding $E_i \hookrightarrow E_{i+1}$ has a prime index. The length of this steps is also bounded by some constant related to $D$. Moreover, $\delta(E_i)\leq \delta(\Lambda)\leq C$ for any $i$ since $E^*_{i} \hookrightarrow \Lambda^*$ and so $Char(E_i)\subset Char(\Lambda)$. Thus we complete the proof of theorem by an induction along each prime index embedding.
\end{proof}

\begin{proof}[Proof of Theorem \ref{thm:finite}]
	Let $W$ be a negative definite, smooth cobordism from $Y_1$ and $Y_2$. If $X$ is a negative definite 4-manifold with the boundary $Y_1$, then $X\cup_{Y_1} W$ is a negative definite 4-manifold bounded by $Y_2$. Moreover, the intersection form of $X$ embeds into the intersection form of $X\cup_{Y_1} W$. By the necessary conditions for a negative definite lattice to be bounded by a rational homology sphere discussed in Section \ref{sec:restrictions}, $\inter(Y_1)$ is a subset of the union of 
	\[\lattice(Q_W,\Gamma; \max_{\mathfrak{t}\in\text{Spin}^c(Y_1)}d(Y_1,\mathfrak{t}), D)\]
	over all integers $D$ dividing $|H_1(Y_1;\mathbb{Z})|$ and $[\Gamma]\in\inter(Y_2)$. Then Theorem \ref{thm:finite} follows from Theorem \ref{thm:alg}. 
\end{proof}

\section{Definite lattices bounded by Seifert fibered rational homology 3-spheres}\label{sec:Seifert}
In this section, we discuss which Seifert fibered 3-manifolds satisfy the condition in Corollary \ref{thm:finite}. In particular, we completely classify spherical 3-manifolds $Y$ such that both $\inter(Y)$ and $\inter(-Y)$ are nonempty, and show that $|\inter(Y)|<\infty$ for any spherical 3-manifold $Y$.

\subsection{Seifert fibered spaces}
Seifert fibered 3-manifolds are a large class of 3-manifolds that contains 6 geometries among Thurston's 8 geometries of 3-manifolds. A Seifert fibered rational homology 3-sphere can be represented by a Seifert form
\[(e_0;(a_1,b_1),\dots,(a_k,b_k)),\]
where $e_0$, $a_i$'s are integers, $b_i$'s are positive integers and $\gcd(a_i,b_i)=1$, and Dehn-surgery diagram of the corresponding 3-manifold is depicted in Figure \ref{fig:Seifert} for the case that $k=3$. Let $M(e_0;(a_1,b_1),(a_2,b_2),\dots,(a_k,b_k))$ denote the corresponding 3-manifold. A Seifert fibered 3-manifold $M(e_0;(a_1,b_1),\dots,(a_k,b_k))$ naturally bounds a 4-manifold constructed by the plumbing diagram in Figure \ref{fig:Seifert}, in which $\alpha_{ij}\in\Z$ are determined by the following Hirzebruch-Jung continued fraction:
\begin{equation*}
	\frac{a_i}{b_i}=[\alpha_{i}^{1},\alpha_{i}^{2},\dots,\alpha_{i}^{l_i}]=\alpha_{i}^{1}-\cfrac{1}{\alpha_{i}^{2}-\cfrac{1}{\dots-\cfrac{1}{\alpha_{i}^{l_i}}}},
\end{equation*}
where $\alpha_{i}^{j}\geq2$ for $1\leq i\leq k$ and $2\leq j\leq l_i$. 

It follows easily from the blow-up procedure and Rolfsen's twist that \[M(e_0;(a_1,b_1),\dots,(\pm1,1),\dots,(a_k,b_k))\] and \[M(e_0\pm1;(a_1,b_1),\dots,\widehat{(\pm1,1)},\dots,(a_k,b_k)),\] and \[M(e_0;(a_1,b_1),\dots,(a_j,b_j),\dots,(a_k,b_k))\] and \[M(e_0+n;(a_1,b_1),\dots,(a_j,b_j-na_j),\dots,(a_k,b_k))\] represent the same homeomorphic 3-manifold respectively: see \cite[Chapter 5.3]{Gompf-Stipsicz}. Hence any Seifert fibered rational homology 3-sphere admits a canonical Seifert form, \[(e_0;(a_1,b_1),(a_2,b_2),\dots,(a_k,b_k))\] such that $a_i>b_i>0$ for all $1\leq i\leq k$. We refer the form by the \emph{normal form} of a Seifert fibered rational homology 3-sphere. We call the 4-manifold, obtained from the normal form of a Seifert fibered 3-manifold, the \emph{canonical plumbed 4-manifold} of the Seifert fibered 3-manifold.

\begin{lemma}
	The intersection form of the corresponding plumbed 4-manifold of a normal form $M(e_0; (a_1,b_1),\dots,(a_k,b_k))$ is negative definite if and only if \[e(M)\coloneqq e_0+\frac{b_1}{a_1}+\dots+\frac{b_k}{a_k}<0.\]
\end{lemma}
\begin{proof}
	This is directly obtained by diagonalizing the corresponding intersection form of the plumbed 4-manifold after extending over $\Q$.
\end{proof}
We refer $e(M)$ by the \emph{Euler number} of a Seifert form $M$. Notice that the Euler number is in fact an invariant for the Seifert fibered 3-manifold $Y$, and $e(-Y)=-e(Y)$.

\begin{proposition}\label{prop:both}
	Any Seifert fibered rational homology 3-sphere can bound positive or negative definite smooth 4-manifolds.
\end{proposition}
\begin{proof}
	Let $Y$ be a Seifert fibered rational homology 3-sphere. By the linking pairing of the surgery diagram in Figure \ref{fig:Seifert}, a rational homology 3-sphere $Y$ has a nonzero Euler number. If $e(Y)<0$, then $Y$ bounds a negative definite 4-manifold by the above lemma. If $e(Y)>0$, then $-Y$ bounds a negative definite 4-manifold since $e(-Y)=-e(Y)<0$. Thus $Y$ bounds a positive definite one.
\end{proof}
\begin{figure}[tb]
	\centering
	\includegraphics[width=0.9\textwidth]{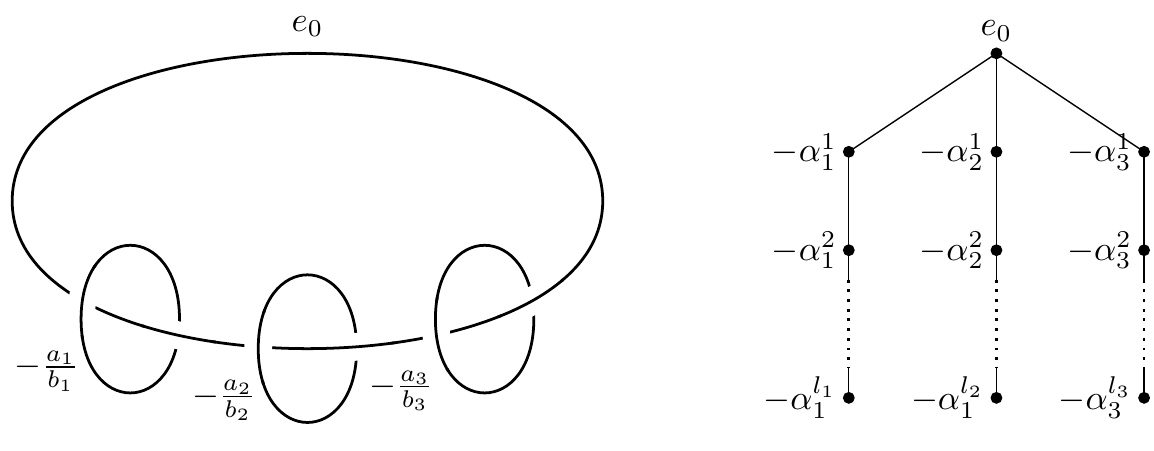}
	\caption{The surgery diagram and plumbing diagram of Seifert manifold $M(e_0; (a_1,b_1),(a_2,b_2),(a_3,b_3))$.}
	\label{fig:Seifert}
\end{figure}

Now we introduce a condition for a Seifert fibered rational homology 3-sphere to bound definite 4-manifolds of both signs.
\begin{proposition}\label{prop:Seifert}
	Let $Y$ be a Seifert fibered rational homology 3-sphere of the normal form 
	\[(e_0; (a_1,b_2),\dots,(a_k,b_k)).\] 
	If $e_0+k\leq0$, then $Y$ bounds both positive and negative definite smooth 4-manifolds, i.e. both $\inter(Y)$ and $\inter(-Y)$ are not empty.
\end{proposition}
\begin{proof}
	Let $Y$ be a Seifert 3-manifold with the normal form $(e_0; (a_1,b_2),\dots,(a_k,b_k))$ such that $e_0+k\leq0$. By the previous proposition, $Y$ bounds a negative definite 4-manifold. To find a positive definite bounding of $Y$, consider the plumbed 4-manifold $X$ corresponding to $-Y\cong M(-e_0-k;(a_1,a_1-b_1),\dots,(a_k,a_k-b_k))$. Note that $b_2^{+}(X)=1$. By blowing up $(e_0+k)$ points on the sphere in $X$ corresponding to the central vertex, we get a sphere with self intersection $0$ in $X\#(e_0+k)\overline{\mathbb{CP}^2}$. By doing a surgery on this sphere, we obtain a desired negative definite 4-manifold. More precisely, we remove the interior of the tubular neighborhood of the sphere, $S^2\times D^2 \subset X\#(e_0+k)\overline{\mathbb{CP}^2}$, and glue $D^3\times S^1$ along the boundary, and it reduces $b_2^+(X\#(e_0+k)\overline{\mathbb{CP}^2})$ by $1$.
\end{proof}

\begin{remark}
	This proposition can be alternatively proved by the fact that these Seifert fibered spaces can be obtained by the branched double covers of $S^3$ along alternating Montesinos links. See \cite[Section 4]{Manolescu-Owens:2007}.
\end{remark}

Note that the condition in Proposition \ref{prop:Seifert} is not a necessary condition. For example, the Brieskorn manifold, \[\Sigma(2, 3, 6n+1)\cong M(-1, ({2},{1}), ({3},{1}), ({6n+1},{1}))\] bounds both negative and positive definite 4-manifolds since $e(M)<0$ and it can be obtained by $(+1)$-surgery of $S^3$ on the $n$-twist knot.

On the other hand, the inequality $e_0+k\leq0$ is sharp since the Brieskorn manifold $\Sigma(2, 3, 5)\cong M(-2, ({2},{1}), ({3}, {2}), ({5},{4}))$, of which $e_0+k=1$, cannot bound any positive definite 4-manifold by the constraint from Donaldson's diagonalization theorem. Note that $\Sigma(2, 3, 5)$ is the Poincar\'e homology sphere $\Sigma$ in Example \ref{ex:poincare}.

\subsection{Spherical 3-manifolds}
A 3-manifold is spherical if it admits a metric of constant curvature $+1$. It is well known that a spherical 3-manifold has a finite fundamental group, and conversely a closed 3-manifold with a finite fundamental group is spherical by the elliptization theorem due to Perelman.

Spherical 3-manifolds are divided into 5-types: $\mathbf{C}$ (cyclic), $\mathbf{D}$ (dihedral), $\mathbf{T}$ (tetrahedral), $\mathbf{O}$ (octahedral) and $\mathbf{I}$ (icosahedral) type, in terms of their fundamental groups. Note that spherical 3-manifolds are Seifert fibered, and their normalized Seifert forms are given as follows, up to the orientation of the manifolds \cite{Seifert:1933}:
\begin{itemize}
	\item Type $\mathbf{C}$; $M(e_0;(a_1,b_1))$
	\item Type $\mathbf{D}$; $M(e_0;(2,1),(2,1),(a_3,b_3))$
	\item Type $\mathbf{T}$; $M(e_0;(2,1),(3,b_2),(3,b_3))$
	\item Type $\mathbf{O}$; $M(e_0;(2,1),(3,b_2),(4,b_3))$
	\item Type $\mathbf{I}$; $M(e_0;(2,1),(3,b_2),(5,b_3))$,
\end{itemize}
where $e_0\leq-2$, $a_i>b_i>0$ and $gcd(a_i,b_i)=1$. Notice that the manifolds are oriented so that their canonical plumbed 4-manifolds are negative definite.

We claim that most of the spherical 3-manifolds can bound smooth definite 4-manifolds of both signs, except the following cases:
\[T_1=M(-2;(2,1),(3,2),(3,2)),\] 
\[O_1=M(-2;(2,1),(3,2),(4,3)),\] 
\[I_1=M(-2;(2,1),(3,2),(5,4)),\] 
and 
\[I_7=M(-2;(2,1),(3,2),(5,3)).\]
Remark that we follow the notations of Bhupal and Ono in \cite{Bhupal-Ono:2012} for this class of 3-manifolds.
\begin{proposition}
	The manifolds $T_1$, $O_1$, $I_1$ and $I_7$ cannot bound a positive definite smooth 4-manifolds.
\end{proposition}
\begin{proof}
	In {\cite[Lemma 3.3]{Lecuona_Lisca_Stein}}, Lecuona and Lisca showed that if $1<\frac{a_1}{b_1}\leq \frac{a_2}{b_2}\leq \frac{a_3}{b_3}$ and $1<\frac{b_2}{a_2}+\frac{b_3}{a_3}$, then the intersection lattice of the plumbing associated to $M(-2;(a_1,b_1),(a_2,b_2),(a_3,b_3))$ cannot be embedded into a negative definite standard lattice. 
	
	Observe that the manifolds, $T_1$, $O_1$, $I_1$ and $I_7$ satisfy the conditions of the lemma. Hence these manifolds cannot bound any positive definite 4-manifolds by the standard argument using Donaldson's theorem.
\end{proof}

\begin{proposition}\label{prop:spherical}
	Any spherical 3-manifolds except $T_1$, $O_1$, $I_1$ and $I_7$ can bound both positive and negative definite smooth 4-manifolds.
\end{proposition}
\begin{proof}
	Since we orient a spherical 3-manifold to bound a natural negative definite 4-manifold, it is enough to construct a positive definite one with the given boundary $Y$. Type-$\mathbf{C}$ manifolds are lens spaces, and it is well known that they bound positive definite 4-manifolds either.
	
	Let $(n,q)$ be a pair of integers such that $1<q< n$ and $\gcd(n,q)=1$, and $\frac{n}{q}=[b,b_1,\dots,b_{r-1},b_r]$. We denote by $D_{n, q}$, the dihedral manifold \[M(-b;(2,1),(2,1),(q,bq-n)).\] The canonical plumbed 4-manifold of $D_{n,q}$ is given in Figure \ref{fig:dihedral}. If $b>2$, there is a positive definite bounding by Proposition \ref{prop:Seifert}. In the case $b=2$, we can check that $-D_{n,q}\cong M(-1;(2,1),(2,1),(q,n-q))$, and the corresponding plumbed 4-manifold $X$ satisfies $b_2^+(X)=1$. As seen in Figure \ref{fig:dihedral2}, we get a 2-sphere with self-intersection 0 after blowing down twice from $X$, and the sphere intersects algebraically twice with the sphere of self-intersection $-c+3$. Hence the sphere with self-intersection 0 is homologically essential, and we obtain a desired negative definite 4-manifold by a surgery along the sphere. 
	
	In the other cases (tetrahedral, octahedral and icosahedral cases), we can apply similar argument except the 4-cases.
\end{proof}

\begin{figure}[!tb]
	\centering
	\includegraphics[width=0.7\textwidth]{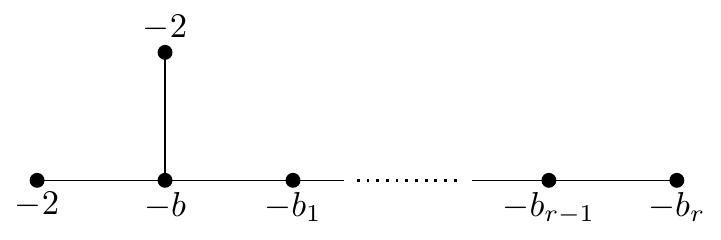}
	\caption{The plumbing graph of the canonical plumbed 4-manifold of $D_{n,q}$, 
		where $\frac{n}{q}=[b,b_1,\dots,b_{r-1},b_r]$.}
	\label{fig:dihedral}
\end{figure}

\begin{figure}[!tb]
	\centering
	\includegraphics[width=0.8\textwidth]{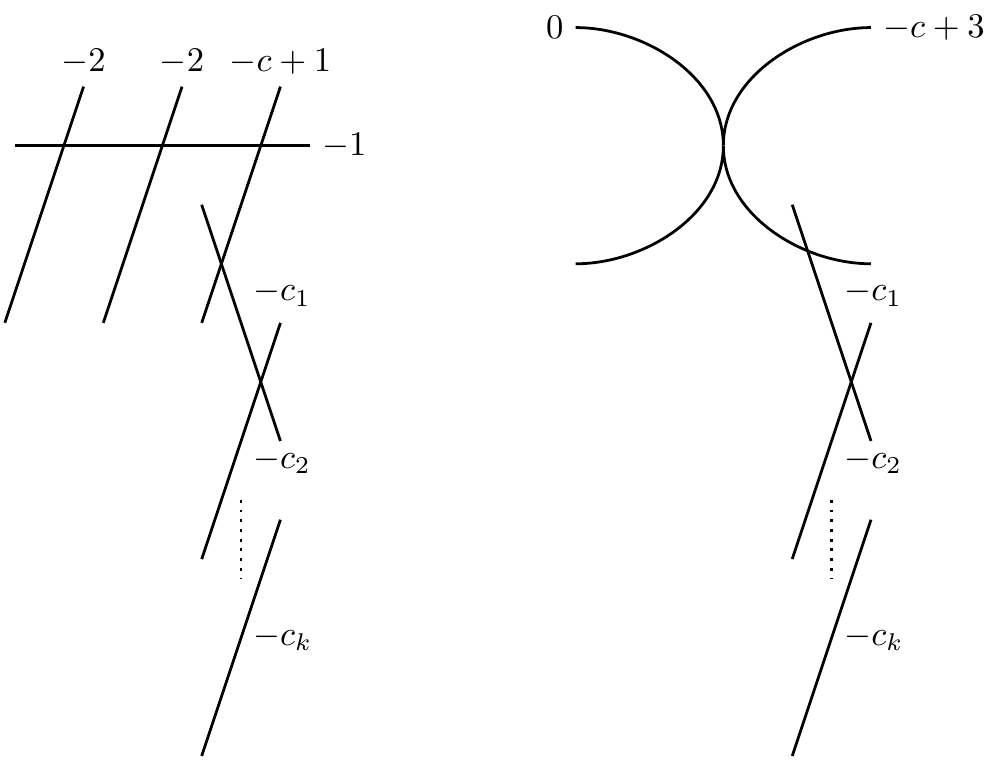}
	\caption{The canonical plumbed 4-manifold of $-D_{n,q}$ and the configuration after blow-down twice,
		where $\frac{n}{n-q}=[c,c_1,\dots,c_k].$}
	\label{fig:dihedral2}
\end{figure}

By Proposition \ref{prop:spherical} and Theorem \ref{thm:finite}, we know that the most of the spherical manifolds have finitely many stable classes of definite lattices to bound them. Finally, we show that the exceptional cases of spherical 3-manifolds also satisfy such finiteness property.

\begin{proof}[Proof of Theorem \ref{thm:spherical}]
	Note that $I_1\cong\Sigma$, and in this case we have $|\inter(\Sigma)|\leq15$ from the lattice theoretic result in \cite{Elkies2} as mentioned in the introduction. We utilize this for the other cases.
	
	Observe that the canonical plumbed 4-manifold $X_{T_1}$ of $T_1$ can be embedded in the canonical plumbed 4-manifold $X_{\Sigma}$ of $\Sigma$. See Figure \ref{fig:em}. Let $W$ be a 4-manifold constructed by removing $X_{T_1}$ from $X_{\Sigma}$. Then $W$ is a negative definite cobordism from $T_1$ to $\Sigma$. The finiteness of $\inter(T_1)$ follows from Theorem \ref{thm:finite}.
	
	\begin{figure}[!tb]
		\centering
		\includegraphics[width=0.8\textwidth]{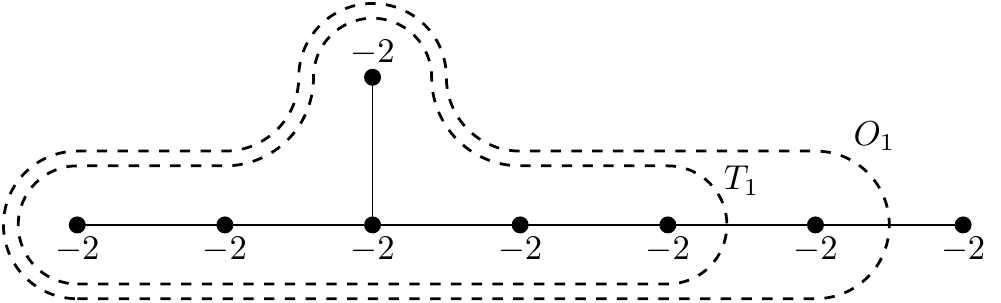}\\
		\vspace{20pt}
		\includegraphics[width=0.8\textwidth]{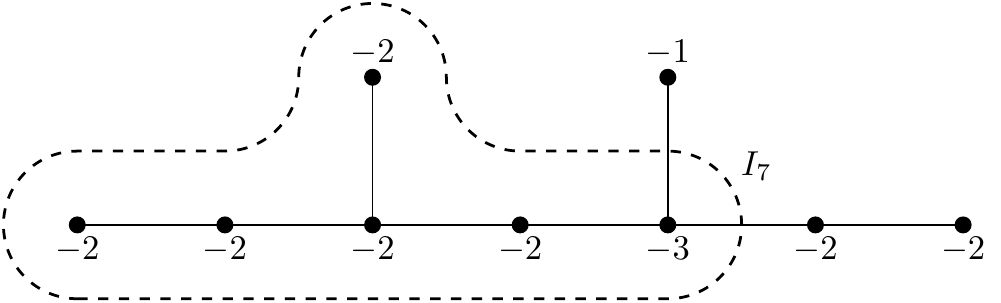}
		\caption{The embedding of the plumbed 4-manifold corresponding to the manifolds $T_1$ and $O_1$, and $I_7$ into $-E_8$-manifold and  $-E_8\#\overline{\mathbb{CP}^2}$, respectively.}
		\label{fig:em}
	\end{figure}
	
	Since the canonical plumbed 4-manifold $O_1$ is also embedded in $X_\Sigma$, we have $|\inter(O_1)|<\infty$. For the manifold $I_7$, we blow up on the sphere in $X_\Sigma$ to contain the canonical plumbed 4-manifold of $I_7$. Then the same argument works to show that $|\inter(I_7)|<\infty$.
\end{proof}

As we mentioned in the introduction, it is known that the 3-manifolds obtained by the double branched cover of quasi-alternating links in $S^3$ bound both positive and negative definite 4-manifolds with trivial $H_1$. Indeed, there are some family of Seifert fibered 3-manifolds that are not obtained by double branched cover on a quasi-alternating link but can be shown to bound definite 4-manifolds of both signs by our result. For example, the dihedral manifolds $D_{n,n-1}$ are such 3-manifolds. 

\begin{proposition}
	The dihedral manifold $D_{n,n-1}$ cannot bound a positive definite smooth 4-manifold with trivial $H_1$, and consequently cannot be obtained by the double branched cover of $S^3$ along a quasi-alternating link in $S^3$.
\end{proposition}
\begin{proof}
	Let $Y$ be $D_{n,n-1}$ manifold, 
	and $X$ be the canonical plumbed 4-manifold of $Y$. If $W$ is a positive definite 4-manifold bounded by $Y$, then, as usual, $Q_X\oplus -Q_W$ embeds into $\langle-1\rangle^{\rk(Q_X)+\rk(Q_W)}$. First observe that an embedding $\iota$ of $Q_X$ to a standard definite lattice is unique, up to the automorphism of the standard definite lattice, as depicted in Figure \ref{fig:dn}, in terms of the standard basis $\{e_1,e_2,\dots,e_{n+1}\}$. For this unique embedding, we have $(\text{Im}\iota)^\perp\cong\langle-1\rangle^{\rk(Q_W)}$.
	
	\begin{figure}[!htb]
		\centering
		\includegraphics{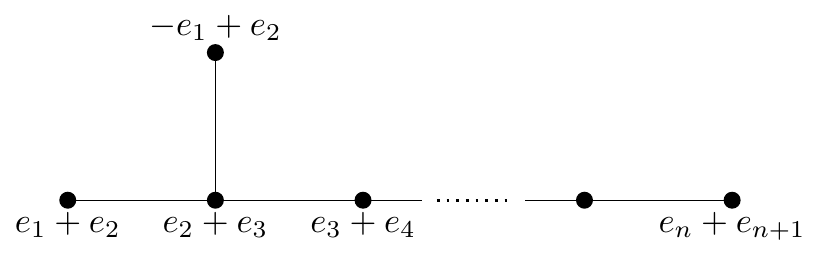}
		\caption{The canonical plumbed 4-manifold of $D_{n,n-1}$.}
		\label{fig:dn}
	\end{figure}
	
	Since $H_1(X; \Z)$ is trivial, $-Q_W$ is isomorphic to $(\text{Im}\iota)^\perp\cong\langle-1\rangle^{\rk(Q_W)}$. Suppose $H_1(W)$ is trivial. Then $H^2(W,\partial W)$ and $H^2(W)$ are torsion free, and $H^2(Y)$ have to be trivial from the following long exact sequence:
	\begin{equation*}
		\dots\rightarrow H^2(W,\partial W)\xrightarrow[Q_W]{\cong} H^2(W)\rightarrow H^2(Y)\rightarrow H^3(W,\partial W)=0
	\end{equation*}
	However, we know that $H^2(Y)$ is non-trivial.
\end{proof}
\begin{remark}
	Recently, all quasi-alternating Montesinos links are completely classified due to Issa in \cite{Issa:2017}. Since any Seifert fibered rational homology 3-sphere is the double branched cover of $S^3$ along a Montesinos link, the above proposition might be followed from his result.
\end{remark}

\section*{Acknowledgments} The authors would like to thank Jongil Park and Ki-Heon Yun for interests on this project and helpful discussions, and Matt Hedden and Brendan Owens for valuable comments on an earlier version of this paper.

\bibliographystyle{alpha}
\bibliography{Intersection_Forms}{}

\begin{thebibliography}{Edm05}

\bibitem[BG18]{Behrens-Golla:2018-1}
S.~Behrens and M.~Golla.
\newblock Heegaard {F}loer correction terms, with a twist.
\newblock {\em Quantum Topol.}, 9(1):1--37, 2018.

\bibitem[BO12]{Bhupal-Ono:2012}
M.~Bhupal and K.~Ono.
\newblock Symplectic fillings of links of quotient surface singularities.
\newblock {\em Nagoya Math. J.}, 207:1--45, 2012.

\bibitem[Boy86]{Boyer:1986-1}
S.~Boyer.
\newblock Simply-connected {$4$}-manifolds with a given boundary.
\newblock {\em Trans. Amer. Math. Soc.}, 298(1):331--357, 1986.

\bibitem[Don83]{Donaldson:1983}
S.~K. Donaldson.
\newblock An application of gauge theory to four-dimensional topology.
\newblock {\em J. Differential Geom.}, 18(2):279--315, 1983.

\bibitem[Don87]{Donaldson:1987}
S.~K. Donaldson.
\newblock The orientation of {Y}ang-{M}ills moduli spaces and {$4$}-manifold
  topology.
\newblock {\em J. Differential Geom.}, 26(3):397--428, 1987.

\bibitem[Edm05]{Edmonds}
A.~L. Edmonds.
\newblock Homology lens spaces in topological 4-manifolds.
\newblock {\em Illinois J. Math.}, 49(3):827--837 (electronic), 2005.

\bibitem[Elk95a]{Elkies}
N.~D. Elkies.
\newblock A characterization of the {${\bf Z}\sp n$} lattice.
\newblock {\em Math. Res. Lett.}, 2(3):321--326, 1995.

\bibitem[Elk95b]{Elkies2}
N.~D. Elkies.
\newblock Lattices and codes with long shadows.
\newblock {\em Math. Res. Lett.}, 2(5):643--651, 1995.

\bibitem[Fre82]{Freedman:1982}
M.~H. Freedman.
\newblock The topology of four-dimensional manifolds.
\newblock {\em J. Differential Geom.}, 17(3):357--453, 1982.

\bibitem[Fr{\o}96]{Froyshov:1996}
K.~A. Fr{\o}yshov.
\newblock The {S}eiberg-{W}itten equations and four-manifolds with boundary.
\newblock {\em Math. Res. Lett.}, 3(3):373--390, 1996.

\bibitem[Gau07]{Gaulter}
M.~Gaulter.
\newblock Characteristic vectors of unimodular lattices which represent two.
\newblock {\em J. Th\'eor. Nombres Bordeaux}, 19(2):405--414, 2007.

\bibitem[GS99]{Gompf-Stipsicz}
R.~E. Gompf and A.~I. Stipsicz.
\newblock {\em {$4$}-manifolds and {K}irby calculus}, volume~20 of {\em
  Graduate Studies in Mathematics}.
\newblock American Mathematical Society, Providence, RI, 1999.

\bibitem[Iss17]{Issa:2017}
A.~Issa.
\newblock The classification of quasi-alternating montesinos links, 2017.
\newblock to appear in Proc. Amer. Math. Soc., available at ar{X}iv:1701.08425.

\bibitem[LL11]{Lecuona_Lisca_Stein}
A.~G. Lecuona and P.~Lisca.
\newblock Stein fillable {S}eifert fibered 3-manifolds.
\newblock {\em Algebr. Geom. Topol.}, 11(2):625--642, 2011.

\bibitem[LR14]{Levine-Ruberman:2014}
A.~S. Levine and D.~Ruberman.
\newblock Generalized {H}eegaard {F}loer correction terms.
\newblock In {\em Proceedings of the {G}\"okova {G}eometry-{T}opology
  {C}onference 2013}, pages 76--96. G\"okova Geometry/Topology Conference
  (GGT), G\"okova, 2014.

\bibitem[MH73]{Milnor-Husemoller:1973}
J.~Milnor and D.~Husemoller.
\newblock {\em Symmetric bilinear forms}.
\newblock Springer-Verlag, New York-Heidelberg, 1973.
\newblock Ergebnisse der Mathematik und ihrer Grenzgebiete, Band 73.

\bibitem[MO07]{Manolescu-Owens:2007}
C.~Manolescu and B.~Owens.
\newblock A concordance invariant from the {F}loer homology of double branched
  covers.
\newblock {\em Int. Math. Res. Not. IMRN}, (20):Art. ID rnm077, 21, 2007.

\bibitem[NV03]{Nebe_Venkov}
G.~Nebe and B.~Venkov.
\newblock Unimodular lattices with long shadow.
\newblock {\em J. Number Theory}, 99(2):307--317, 2003.

\bibitem[OS03]{Ozsvath-Szabo:2003}
P.~Ozsv{\'a}th and Z.~Szab{\'o}.
\newblock Absolutely graded {F}loer homologies and intersection forms for
  four-manifolds with boundary.
\newblock {\em Adv. Math.}, 173(2):179--261, 2003.

\bibitem[OS05a]{Owens_Strle_survey}
B.~Owens and S.~Strle.
\newblock Definite manifolds bounded by rational homology three spheres.
\newblock In {\em Geometry and topology of manifolds}, volume~47 of {\em Fields
  Inst. Commun.}, pages 243--252. Amer. Math. Soc., Providence, RI, 2005.

\bibitem[OS05b]{Ozsvath-Szabo:2005-1}
P.~Ozsv\'ath and Z.~Szab\'o.
\newblock On the {H}eegaard {F}loer homology of branched double-covers.
\newblock {\em Adv. Math.}, 194(1):1--33, 2005.

\bibitem[OS06]{Owens_Strle3}
B.~Owens and S.~Strle.
\newblock Rational homology spheres and the four-ball genus of knots.
\newblock {\em Adv. Math.}, 200(1):196--216, 2006.

\bibitem[OS12a]{Owens_Strle}
B.~Owens and S.~Strle.
\newblock A characterization of the {$\Bbb Z\sp n\oplus\Bbb Z(\delta)$} lattice
  and definite nonunimodular intersection forms.
\newblock {\em Amer. J. Math.}, 134(4):891--913, 2012.

\bibitem[OS12b]{Owens-Strle:2012-1}
B.~Owens and S.~Strle.
\newblock Dehn surgeries and negative-definite four-manifolds.
\newblock {\em Selecta Math. (N.S.)}, 18(4):839--854, 2012.

\bibitem[Sei33]{Seifert:1933}
H.~Seifert.
\newblock Topologie {D}reidimensionaler {G}efaserter {R}\"aume.
\newblock {\em Acta Math.}, 60(1):147--238, 1933.

\end{thebibliography}
\end{document}